\documentclass[12pt]{article}
\usepackage[doc]{optional}
\usepackage{enumitem}
\usepackage{mathtools,esvect}
\usepackage{cite}
\usepackage{color}
\usepackage{float}
\usepackage{subcaption}
\usepackage{soul}
\usepackage{graphicx}
\definecolor{labelkey}{rgb}{0,0.08,0.45}
\definecolor{refkey}{rgb}{0,0.6,0.0}

\usepackage[left]{lineno}
\usepackage{blindtext}

\usepackage{mathpazo}
\usepackage{amsmath}
\usepackage{amssymb}
\usepackage{theorem}
\usepackage{fancyhdr}

\usepackage[margin=1.0in]{geometry}

\newcommand{\fenv}[1]%
{\ensuremath{\,\overrightarrow{\operatorname{env}}_{#1}}}
\newcommand{\benv}[1]%
{\ensuremath{\,\overleftarrow{\operatorname{env}}_{#1}}}

\newcommand{\RR}{\ensuremath{\mathbb R}}

\newcommand{\NN}{\ensuremath{\mathbb N}}

\newcommand{\aff}{\ensuremath{\operatorname{aff}}}

\newcommand{\Fix}{\ensuremath{\operatorname{Fix}}}

\newcommand{\Id}{\ensuremath{\operatorname{Id}}}

{\begin{list}{}{%
\settowidth{\labelwidth}{\textrm{#1~}}%
\setlength{\leftmargin}{\labelwidth+\labelsep}}}
{\end{list}}
\newtheorem{theorem}{Theorem}[section]
\newtheorem{lemma}[theorem]{Lemma}

\theoremstyle{plain}{\theorembodyfont{\rmfamily}
}
\theoremstyle{plain}{\theorembodyfont{\rmfamily}
}
\theoremstyle{plain}{\theorembodyfont{\rmfamily}
}
\theoremstyle{plain}{\theorembodyfont{\rmfamily}
}

\theoremstyle{plain}{\theorembodyfont{\rmfamily}
}

\def\doi{DOI}

\newcounter{count}

\allowdisplaybreaks

\begin{document}

\title{\textrm{On the linear convergence of the circumcentered--reflection method}}
\author{R. Behling\thanks{Department of Exact Sciences, Federal University of Santa Catarina.
Blumenau, SC -- 88040-900, Brazil. E-mail:
\texttt{\{roger.behling,l.r.santos\}@ufsc.br.}}~~~~~
J.Y.\ Bello Cruz\thanks{Department of Mathematical Sciences, Northern Illinois University. Watson Hall 366, DeKalb, IL -- 60115-6117, USA. E-mail:
\texttt{yunierbello@niu.edu.}}~~~~~L.-R.\ dos Santos\footnotemark[1] }

\date{\today}

\maketitle \thispagestyle{fancy}

\vskip 10mm

\begin{abstract}
\noindent

In order to accelerate the Douglas--Rachford method we  recently developed the circumcentered--reflection method, which provides the closest iterate to the solution among all points relying on successive reflections, for the best approximation problem related to two affine subspaces. We now prove that this is still the case when considering a family of finitely many affine subspaces. This property yields linear convergence  and incites embedding of circumcenters within classical reflection and projection based methods for more general feasibility problems.

\textbf{Keywords:} reflection; projection; best approximation problem; Douglas--Rachford method.
\end{abstract}







\section{Introduction}
We consider the fundamental feasibility problem of projecting onto the intersection of (affine) subspaces, also referred to as \textit{linear best approximation problem}. Let $\{U_i\}_{i\in I}$ be a family of finitely many subspaces in $\RR^n$ (all nonempty) with $I\coloneqq \{1,2,3,\ldots,m\}$ and $m$ fixed (we require no relation between $n$ and $m$). Their nonempty intersection is denoted by $S\coloneqq \cap_{i\in I} U_i$ and the problem we are interested in consists of projecting a given point $x\in\RR^n$ onto $S$,  that is, 
\begin{equation}\label{Problem1}
\text{Find } \bar x\in S\text{ such that } \|\bar x-x\|=\min_{s\in S} \|s-x\|.
\end{equation}
It is well known that this problem has a unique solution $\bar x$ and that $\bar x=P_S(x)$ if, and only if, $x-\bar x\in S^\perp$, {\em i.e.},
$\langle x-\bar x, s\rangle =0$ for all $s\in S$, where $\langle \cdot, \cdot\rangle$ denotes the scalar product in $\RR^n$. Also, $\|\cdot\|$ is the induced Euclidean vector or matrix norm, and $P_S$ denotes the orthogonal projection onto $S$. 

Reflection and projection type methods are famous elementary tools for solving convex feasibility inclusions, in particular problem \eqref{Problem1}. They remain trendy due to good balance between efficiency and easy implementation, even in certain nonconvex settings~\cite{Bauschke:2006ej}. Broadly known are schemes based on the method of alternating projections (MAP)~\cite[Chapter~9]{Deutsch:2001fl} (which is also known as von Neumann's alternating projection algorithm)  and the Douglas--Rachford method (DRM)~\cite{Douglas:1956kk,BCNPW14} (or averaged--reflection method). DRM handles satisfactorily some highly relevant types of feasibility problems even nonconvex ones, nevertheless, it is proven to possibly fail for problem \eqref{Problem1} if $m>2$ (see \cite[Example $2.1$]{ABTcomb}). Variants for the many set case have therefore been developed, \emph{e.g.}, the cyclically anchored Douglas--Rachford algorithm (CADRA)~\cite{Bauschke:2015df} and the cyclic Douglas--Rachford method (CyDRM)~\cite{Borwein:2014ka}. MAP, on the other hand, manages to converge to a solution of problem~\eqref{Problem1}, at the price of frequent slow convergence, though (see, {\em e.g.},~\cite{BCNPW15}). 

The circumcentered--reflection method (CRM) proposed here is, to a certain extend, meant to minimize the inherent \textit{zig-zag} behavior of sequences generated by standard reflection and projection type methods. It may be seen as a generalization of the circumcentered-Douglas--Rachford method proposed in \cite{Behling:2017da} for accelerating the convergence of the DRM for \eqref{Problem1} with only two affine subspaces. This novel and simple idea relies on a necessary optimality condition arising from norm preserving after reflecting onto subspaces. In this paper, we show linear convergence of CRM for solving problem \eqref{Problem1}, which improves the results from \cite{Behling:2017da} and instigates theoretical and numerical investigation in more general contexts. 

After an ordered round of successive reflections onto the subspaces $U_i$'s, CRM chooses the new iterate by means of equidistance to the reflected points, justifying the use of the term \textit{circumcenter}. More precisely,
at a given point $x\in\RR^n$, CRM generates a next iterate $C(x)\in\RR^n$ with the following two properties:

\noindent {\bf (a)} $C(x)$ belongs to the affine subspace 
\begin{align*}
W_x\coloneqq \aff\{x,R_{U_1}(x),R_{U_2}R_{U_1}(x), \ldots, R_{U_m} \cdots R_{U_2}R_{U_1}(x)\};
\end{align*}
\noindent {\bf (b)}  $C(x)$ is equidistant to the points $x$, $R_{U_1}(x)$, $R_{U_2}R_{U_1}(x)$, \ldots, $R_{U_m} \cdots R_{U_2}R_{U_1}(x)$,

\noindent where $R_{U_i}$ denotes the reflection operator onto $U_i$, characterized by satisfying $R_{U_i}=2P_{U_i}-\Id$, with $\Id$ being the identity operator. Along the text, we use the abbreviation $x^{(i)}\coloneqq R_{U_i}\cdots R_{U_2}R_{U_1}(x)$ with $i\in I$.

We show that the (non-linear) CRM operator $C:\RR^n\rightarrow \RR^n$ is well defined and, for any $x\in\RR^n$, $C(x)\in\RR^n$ has the optimal property of being the point in $W_x$ lying closest to $S$. The computation of the circumcenter $C(x)$ requires the resolution of a suitable $m\times m$ linear system of equations, being therefore computable at low cost if $m$ is not too large. Our main theorem states that problem \eqref{Problem1} is solved by fixed points of the operator $C$. Moreover, the sequence $\{C^k(x)\}$ converges linearly to $P_S(x)$ and it does, at least, as fast as any reflection based method.

This paper is organized as follows. A study on a suitable averaged linear operator is carried out in Section 2, allowing us to establish linear convergence of CRM in Section 3. Section 4 presents a discussion on future work.
\section{On an auxiliary linear operator}

In this section we will introduce the convenient operator $A:\RR^n\rightarrow\RR^n$, with 

\[A(x)\coloneqq\frac{1}{m}\sum_{i=1}^{m}A_i(x),\text{ where}\]  
\[A_1  \coloneqq \frac{1}{2}(\Id+P_{U_1}),\text{ and } A_i \coloneqq \frac{1}{2}(\Id+P_{U_i}R_{U_{i-1}}\cdots R_{U_1}),
\]
for $i=2,\ldots, m$.

We list next several properties of $A$, some of which will be used in our convergence analysis in the next section. For this, we recall that an operator $T:\RR^n\rightarrow \RR^n$ is called $\alpha$-averaged, with $\alpha \in (0,1)$, if there exists a linear non-expansive operator $G$, so that $T=\alpha \Id +(1-\alpha)G$. If this is true for $\alpha=1/2$, $T$ is called firmly non-expansive~\cite[Remark 4.34(iii)]{BC2011}.

\begin{lemma}\label{PropertiesA}
Let $A:\RR^n\rightarrow \RR^n$ be as above. Then:
\begin{description}
\item[(i)] $A$ is linear and $\tfrac{1}{2}$-averaged, {i.e.}, firmly non-expansive;
\item[(ii)] $\Fix A\coloneqq \{x\in\RR^n|A(x)=x\}=S$;
\item[(iii)] For all $x\in\RR^n$, $A(x)\in W_x$;
\item[(iv)] There exists $r_A\in [0,1)$ such that for all $x\in\RR^n$ it holds that $\|A(x)-P_S(x)\|\leq r_A \|x-P_S(x)\|$;
\item[(v)] Let $x\in\RR^n$ be given. Then, $P_S(A^k(x)) = P_S(x)$ for all $k\in\NN$;
\item [(vi)] Let $x\in\RR^n$ be given. The sequence $\{A^k(x)\}_{k\in\NN}$ converges to $P_S(x)$ with the linear rate $r_A$.
\end{description}
\end{lemma}
\begin{proof}
\noindent {\bf (i)} Note that each of the operators $A_i$ $(i=1,2,\ldots,m)$ is the mean between the identity $\Id$ and the composition of non-expasive operators (reflections and projection) $P_{U_{i}}R_{U_{i-1}}\cdots R_{U_{1}}$ $(i=1,2,\ldots,m)$. The factor $\frac{1}{2}$ serves to guarantee $\tfrac{1}{2}$-averaged of the $A_i$'s. By using  \cite[Proposition 4.42]{BC2011}, $A$ is $\tfrac{1}{2}$-averaged, which is equivalent to $A$ being firmly non-expansive.

\noindent {\bf (ii)}  If $x\in S$ then $A_i(x)=x$ and hence $A(x)=x$. Conversely, it is easy to see that  $\Fix A_i=U_i$. So, if $x=A(x)=\tfrac{1}{m}	(A_1(x)+A_2(x)+\cdots+A_m(x))$ then 
\begin{align}
\nonumber
\label{norm-fix}
\|x\|&=\|A(x)\|=\frac{1}{m}\left\|A_1(x)+A_2(x)+\cdots+A_m(x)\right\| \\ 
&\le\frac{1}{m}\left(\|A_1(x)\|+\|A_2(x)\|+\cdots+\|A_m(x)\|\right).
\end{align} 
Now, if there exists an index $j\in I$ such that $x\notin \Fix A_j=U_j$,  we get $\|A_j(x)\|<\|x\|$ by using the firmly non-expansiveness of $A_i$ $(i=1,2,\ldots,m)$. Hence, \eqref{norm-fix} leads us to a contradiction. Thus, $x\in \Fix A_i=U_i$ for all $(i=1,2,\ldots,m)$.

\noindent {\bf (iii)} This item follows by rewriting \[A_1(x)=\displaystyle \frac{1}{2}\left(x+\frac{x^{(1)}+x}{2}\right)\] and  \[A_i(x)=\frac{1}{2}\left(x+\frac{x^{(i)}+x^{(i-1)}}{2}\right)\] for $i=2,\ldots,m$. In this way, we clearly get that $A(x)\coloneqq\frac{1}{m}\sum_{i=1}^{m}A_i(x)$ is a linear combination of the points $x,x^{(1)}, x^{(2)}, \ldots, x^{(m)}$, \textit{i.e.}, $A(x)\in W_x$.

\noindent {\bf (iv)} If $x\in S$, the statement is trivial. By using  \cite[Proposition 4.35]{BC2011} and Lemma \ref{PropertiesA}(i), we get, for any $x\in \RR^n$ and $s\in S$,
\begin{equation}\label{fejer*}\|A(x)-s\|^2\le \|x-s\|^2-\|x-A(x)\|^2.
\end{equation}
Also, for any $x\in \RR^n$, $x=P_S(x)+P_{S^\perp}(x)=P_S(x)+(x-P_S(x))$ and 
so \begin{align*}
	&x - A(x)  = P_S(x) + x -P_S(x)-A(P_S(x) + x -P_S(x)) \\ 
			  &= (x-P_S(x))-A(x-P_S(x))) + P_S(x)-A(P_S(x))\\ 
			  &= (x-P_S(x))-A(x-P_S(x))).
\end{align*}
since $\Fix A = S$. Thus, it follows from \eqref{fejer*} with $s=P_S(x)$ that 
\begin{align*}\nonumber\|A(x)-P_S(x)\|^2   \le& \|x-P_S(x)\|^2\\&-\|(x-P_S(x))-A(x-P_S(x))\|^2 
\\ <&\|x-P_S(x)\|^2,
\end{align*} 
where the last strict inequality holds for any $x\notin  S$. Furthermore, 
\begin{align}\label{fejer1}
\nonumber
\|A(x)-P_S(x)\| & = \|A(x)-A(P_S(x))\| \\\nonumber  & = \|A(x-P_S(x))\|\\ &\le r_A \|x-P_S(x)\|,
\end{align}
 where $r_A\coloneqq \sup_{y\in S^\perp, \; \|y\|=1} \|Ay\|$ because $x-P_S(x)\in S^\perp$. Note that $r_A<1$, otherwise we get $\bar y\in S^\perp$ such that  $\|A\bar y\|=\|\bar y\|=1$ (by Weierstrass Theorem), contradicting the firmly non-expansiveness of $A$ in \eqref{fejer*} and the fact that $\bar y\notin \Fix A=S$.  

\noindent {\bf (v)} Fix $k\in \NN$, $\alpha \in \RR$, and define $s_\alpha \coloneqq  \alpha P_S(x) +(1-\alpha)P_S(A^k(x))$. Since $S$ is a subspace, $s_\alpha \in S$, and it therefore follows from \eqref{fejer*} and Pythagoras that
\begin{align*}
&\alpha^2 \| P_S(A^k(x))-P_S(x)\|^2 =\| P_S(A^k(x))-s_\alpha\|^2 \\ 
&\le  \|A^k(x) - P_S(A^k(x))\|^2 + \|P_S(A^k(x)) - s_\alpha\|^2\\
&= \|A^k(x) - s_\alpha\|^2\le  \|x - s_\alpha\|^2 \\
&= \|x - P_S(x)\|^2 + \|P_S(x) - s_\alpha\|^2\\
&= {\rm dist}(x,S)^2 + (1 - \alpha)^2\|P_S(A^k(x)) - P_S(x)\|^2.
\end{align*}
Thus, $(2\alpha - 1)\|P_S(A^k(x)) - P_S(x) \|^2 \le  {\rm dist}(x,S)^2$ and, letting $\alpha\to +\infty$, we get $P_S(A^k(x)) = P_S(x)$.

\noindent {\bf (vi)} The convergence of $\{A^k(x)\}$ is consequence of \cite[Proposition 5.16]{BC2011}. This convergence is linear with linear rate $r_A$ given in (iv) because \cite[Corollary 2.8]{Bauschke:2017hz} and (iv). Finally, \cite[ Corollary 5.8]{BC2011} and (v) ensure that the underlying sequence  converges to $P_S(x)$. 
\end{proof}

\section{Convergence analysis of the circumcentered--reflection method}
In order to prove the linear convergence of CRM for problem \eqref{Problem1}, our first lemma not only establishes the existence and uniqueness of the circumcenter, but  also characterizes $C(x)$ as the projection of any $s\in S$ onto $W_x=\aff\{x,x^{(1)}, x^{(2)}, \ldots, x^{(m)}\}$.

\begin{lemma}\label{ProjectionOntoAffReflections}
For any $s\in S$, $x\in\RR^n$, it holds that $C(x)=P_{W_x}(s)$.
\end{lemma}
\begin{proof}
Let $x\in \RR^n$ and $s\in S$ be arbitrary but fixed. Also, let $p$ denote the projection of $s$ onto the nonempty affine space $W_x$. The definition of reflections, Pythagoras and the assumption that $s\in S$ trivially imply the following optimality condition 
\begin{equation}\label{NormPreservReflect}
\|x-s\|=\|x^{(1)}-s\|=\|x^{(2)}-s\|=\cdots=\|x^{(m)}-s\|.
\end{equation}
Note also that $p=P_{W_x}(s)$ allows us to use Pythagoras for each $x^{(i)}$ as follows
\begin{align*}
\|x-s\|^2  &= \|p-s\|^2+\|x-p\|^2, \\ \|x^{(i)}-s\|^2  &= \|p-s\|^2+\|x^{(i)}-p\|^2, \quad i=1,\ldots, m.
\end{align*}						
Clearly, these equalities combined with \eqref{NormPreservReflect} give us
$$
\|x-p\|=\|x^{(1)}-p\|=\|x^{(2)}-p\|=\cdots=\|x^{(m)}-p\|,
$$
that is, $p=P_{W_x}(s)$ is a circumcenter with respect to $x, x^{(1)}, x^{(2)}, \ldots, x^{(m)}$ in $W_x$. Regarding uniqueness, it is a straightforward consequence of uniqueness of centers of closed balls in $\RR^{q(x)}$, with $q(x)\coloneqq \textrm{dim}(W_x)\leq \min\{m,n\}$.
\end{proof}

Before moving on with the convergence analysis, let us explain how $C(x)$ can be computed. $C(x)$ is characterized by being a linear combination of $x$ and the $x^{(i)}$'s (in order to be in $W_x$) with the additional property that its projection onto each closed segment $[x^{(i)},x]$ defines a bisector, that is, the projections of $C(x)$ onto those segments lie on their midpoint, to ensure equidistance. Hence, by anchoring the problem at $x$, we want the conditions $C(x)-x=\sum_{j=1}^{m}\alpha_i (x^{(j)}-x)$ and $P_{\textrm{span}\{x^{(i)}-x\}}(C(x)-x)=\frac{1}{2}(x^{(i)}-x)$, for all $i=1, \ldots, m$, to hold. This yields the solvable $m\times m$ linear system in $\alpha\in\RR^m$ whose $i$-th row is given by
\[
 \sum_{j=1}^{m}\alpha_j\langle x^{(j)}-x, x^{(i)}-x\rangle=\frac{1}{2}\|x^{(i)}-x\|^2.
\] 
$C(x)$ outcomes univocally from this linear system. Uniqueness in $\alpha$, however, depends on linear independence of the vectors $x^{(i)}-x$, which is not always the case.


\begin{lemma}\label{LinearGainCircumcenter}
For all $x\in\RR^n$, it holds that $P_S(C^k(x))=P_S(x)$ for all $k\in \NN$ and $\|C(x)-P_S(x)\|\leq r_A \|x-P_S(x)\|$, where $r_A\in [0,1)$ is as in Lemma~\ref{PropertiesA}.
\end{lemma}
\begin{proof}
Combining $C(x)=P_{W_x}(s)$ from Lemma~\ref{ProjectionOntoAffReflections}, $A(x)\in W_x$ from Lemma \ref{PropertiesA}(iii) and $\|A(x)-P_S(x)\|\leq r_A \|x-P_S(x)\|$ from Lemma \ref{PropertiesA}(iv) yields $$\|C(x)-P_S(x)\|\leq \|A(x)-P_S(x)\|\leq r_A \|x-P_S(x)\|.$$ Note that $P_S(C^k(x))=P_S(x)$ follows easily by induction once $P_S(C(x))=P_S(x)$ is proven. Thus, let us prove the latter. Using again Lemma~\ref{ProjectionOntoAffReflections} we can employ Pythagoras for two suitable triangles
\begin{gather*}
\|P_S(C(x))-C(x)\|^2=\|P_S(C(x))-x\|^2 -\|C(x)-x\|^2,\\
\|P_S(x)-C(x)\|^2=\|P_S(x)-x\|^2 -\|C(x)-x\|^2.
\end{gather*}
These equalities together with $\|P_S(x)-x\|\leq \|P_S(C(x))-x\|$ imply that $\|P_S(x)-C(x)\|\leq \|P_S(C(x))-C(x)\|,$ proving $P_S(C^k(x))=P_S(x)$.
\end{proof}

The previous lemma yields linear convergence of CRM for solving problem \eqref{Problem1}.

\begin{theorem}\label{the:CRMconvergence}
Let $x\in\RR^n$ be given. Then, the {\rm CRM} sequence $\{C^k(x)\}$ converges to $P_S(x)$ with the linear rate $r_A\in [0,1)$, given in Lemma \ref{LinearGainCircumcenter}.
\end{theorem}
\begin{proof} Given $x\in\RR^n$ and for all $k\in\NN$, Lemma \ref{LinearGainCircumcenter} provides 
\[
\begin{aligned}
\|C^{k+1}(x)-P_S(x)\| & =\|C(C^k(x))-P_S(C^k(x))\| \\
				    & \leq r_A \|C^k(x)-P_S(C^k(x))\| \\ 
				    & =r_A\|C^k(x)-P_S(x)\|,
\end{aligned} 
\]
which implies $\|C^{k}(x)-P_S(x)\|\le r_A^{k}\|x-P_S(x)\|$.
\end{proof}

Theorem~\ref{the:CRMconvergence} can be easily extended to affine subspaces with nonempty intersection~\cite[Corol.~3]{Behling:2017da}.
\section{Outlook}\label{FinalRemarks}

In this paper we proved linear convergence of the circumcentered--reflection method CRM for best approximations problems related to a family of finitely many (affine) subspaces in $\RR^n$. This result is the basis for further and more general investigation. In this sense, we list now a few directions of future work.

One of the first questions that arises is whether one can derive convergence of CRM in Hilbert spaces. Remember that, in infinite dimension, the $U_i$'s are not necessarily closed. This alone suggests closedness assumptions in this setting, as we need welldefinedness of reflections/projections onto each $U_i$ and $S$. Less obvious than this remark, is that it might be worth monitoring the sequence $\{C^k(P_{U_j}(x))\}$ instead of $\{C^k(x)\}$, where $j$ is any index between $1$ and $m$. Note that Pythagoras guarantees that the best solution to $x$ and $P_{U_j}(x)$ coincide, that is, $P_S(x)=P_S(P_{U_j}(x))$. The advantage of enforcing such an initialization is due to the fact that $\{C^k(P_{U_j}(x))\}$ lies in $U_1+U_2+\cdots +U_m$, while $\{C^k(x)\}$ may \textit{zig-zag} in a larger region of the corresponding Hilbert space. Generating $\{C^k(P_{U_j}(x))\}$ instead of $\{C^k(x)\}$ is a point to be taken into account even in finite dimensional spaces.

Obviously, extensive numerical experiments on CRM and variants have to be conducted, expectantly with results as favorable as in \cite{Behling:2017da}, where $m=2$ was considered. We intend to compare CRM with prominent methods for solving problem \eqref{Problem1}, \textit{e.g.}, MAP, CADRA, CyDRM, Cimmino method~\cite{Cimmino:1938tp}. Our educated guess is that these methods  could benefit numerically from suitably incorporating circumcentering based techniques. Also, the possibility of taking pairwise or $q$-wise ($q<m$) circumcenters is in the scope of our future studies, as well as deriving the sharp rate for CRM.  Positive experiments together with its geometrical appeal, may reveal CRM as an option for solving highly important feasibility problems (even nonconvex) involving (affine) subspaces, with very promising behavior in several applications, {\em e.g.}, the matrix completion problem~\cite{AragonArtacho:2014id},  the basis pursuit problem~\cite{Demanet:2016fj} and the nonconvex sparse affine feasibility problem~\cite{Hesse:2014gi}. 

The most relevant and nontrivial task, perhaps, is to extend the theory of CRM to certain inconsistent affine~\cite{Bauschke:2016bi,Bauschke:2016jw} and non-affine settings~\cite{Borwein:2015vm}, which has already been established for DRM and its variants. Among them, we are particularly interested in convex feasibility problems as in~\cite{Bauschke:2006ej} and general convex inclusions ({\em e.g.} \cite{Svaiter:2011fb,Bauschke:2004fj,Bauschke:2016bt}). Simple examples, however, show that the circumcenter $C(x)$ may not be defined for general convex sets at some ``pathological'' points $x$. Whenever this happens, our guess is that our operator $A$ in Section 2 provides a good replacement for $C(x)$.\\

\noindent {\bf Acknowledgments:} We thank the anonymous referees for their valuable suggestions. JYBC was partially supported by a startup research grant of Northern Illinois University.  



{\small

}


\end{document}